\documentclass{conm-p-l}

\usepackage{latexsym}
\usepackage{amsmath}
\usepackage{amsthm}
\usepackage{amssymb}
\usepackage{amsfonts}
\usepackage{fancyhdr}
\usepackage{comment}
\usepackage{graphicx}

\usepackage{mathrsfs}                           

\theoremstyle{definition}
\newtheorem{thm}{Theorem}[section]

\newtheorem{cor}[thm]{Corollary}
\newtheorem{lemma}[thm]{Lemma}

\newtheorem{definition}{Definition}

\numberwithin{equation}{section}

\newcommand{\R}{\mathbb{R}}

\newcommand{\p}{\partial}
\newcommand{\ol}{\overline}

\newcommand{\Trg}[2]{~\mbox{Tr}_{#2}\left(#1\right)}

\newcommand{\tr}[1]{~\mbox{tr}\left(#1\right)}
\renewcommand{\l}{\left}
\renewcommand{\r}{\right}

\renewcommand{\to}{\rightarrow}
\newcommand{\To}{\Rightarrow}

\newcommand{\M}{\mathcal{M}}
\newcommand{\N}{\mathcal{N}}

\renewcommand{\exp}[1]{~\mbox{exp}\left(#1\right)}
\newcommand{\e}{\epsilon}
\renewcommand{\d}{\delta}
\newcommand{\F}{\mathcal{F}}
\newcommand{\g}{\gamma}
\renewcommand{\p}{\frac{1}{2}}

\newcommand{\wt}{\widetilde}

\title[Optimal Riemannian metrics and an ergodic theorem in NPC space]{Optimal Riemannian metric for a volumorphism and a mean ergodic theorem in complete global Alexandrov nonpositively curved spaces}

\author{Tony Liimatainen}
\address{Department of Mathematics and Systems Analysis, Aalto University, P.O. Box 11100 FI-00076 Aalto Finland, +358 9 470 23044}
\email{tony.liimatainen@aalto.fi}

\date{February 14, 2012}
\begin{document}
\begin{abstract}
In this paper we give a natural condition for when a volumorphism on a Riemannian manifold $(M,g)$ is actually an isometry with respect to some other, optimal, Riemannian metric $h$. We consider the natural action of volumorphisms on the space $\M_\mu^s$ of all Riemannian metrics of Sobolev class $H^s$, $s>n/2$, with a fixed volume form $\mu$. An optimal Riemannian metric, for a given volumorphism, is a fixed point of this action in a certain complete metric space containing $\M_\mu^s$ as an isometrically embedded subset. We show that a fixed point exists if the orbit of the action is bounded. We also generalize a mean ergodic theorem and a fixed point theorem to the nonlinear setting of complete global Alexandrov nonpositive curvature spaces.
\end{abstract}
\maketitle

\section{Introduction}
In this paper we consider the problem of finding an invariant Riemannian metric for a mapping preserving a given volume form on a Riemannian manifold. The setup of the problem is the following.

Let $M$ be a smooth, closed and oriented finite dimensional manifold. Then the set of all Riemannian metrics $\M$ can be considered as an infinite dimensional manifold. Its tangent vectors can be given an $L^2$ inner product~\cite{Ebin, Clarke} and thus computing the curvature of $\M$ makes sense. The sectional curvature of $\M$ is nonpositive, but $\M$ is not geodesically complete~\cite{Freed, Clarke}.

Instead of $\M$ we consider a submanifold $\M_\mu$ of $\M$ consisting of Riemannian metrics which all have the same given volume form $\mu$. The submanifold $\M_\mu$ is an infinite dimensional globally symmetric space. The exponential mapping is a diffeomorphism, it has nonpositive curvature and any two points can be joined by a unique geodesic. For details of these statements, see~\cite{Ebin, Freed, Clarke}.

Diffeomorphisms preserving a given volume form are called volumorphisms. They act naturally by pullback on $\M_\mu$ and the action is isometric. Let $g$ be some Riemannian metric on $M$ whose induced volume form is $\mu$ and $\phi$ a volumorphism of $M$. Tangent spheres defined by $g$ are mapped to tangent ellipsoids by the pushforward of $\phi$. The sphere and its image under the pushforward of $\phi$, the ellipsoid, have the same volume due to the fact that $\phi$ is a volumorphism, but we have no control on how much the sphere is distorted.

We ask the following question. If the tangent spheres are boundedly distorted under the iterations of the pushforward of $\phi$, is there another Riemannian metric, say $h$, such that tangent spheres defined by $h$ are mapped to spheres? Since in this case the mapping $\phi$ considered as a mapping $(M,h)\to (M,h)$ would be an isometry, the Riemannian metric $h$ is in this sense an optimal Riemannian metric for $\phi$. A similar idea has been used in the study of quasiconformal mappings in~\cite{Tukia, Iwaniec}.

We formalise the idea explained above and ask: ``If the action of a volumorphism on $\M_\mu$ has a bounded orbit, is there a fixed point of this action?''. We consider the space of Riemannian metrics of Sobolev class $H^s$ and assume that the mapping $\phi$ is of Sobolev class $H^{s+1}$, $s>n/2$. We show that the answer to the question above is affirmative if we allow the possibility that the fixed point is of lower regularity than $H^s$. 

A fixed point of the action of $\phi$ on $\M_\mu^s$, if it exists, will belong to a metric space $(X,\d)$ of $\mu$-a.e. positive definite symmetric $(0,2)$-tensor fields with volume form agreeing with $\mu$ a.e. The elements of $X$ are also assumed to satisfy a certain natural integrability condition. As a metric space, $(X,\d)$ is a complete global Alexandrov nonpositive curvature space containing $(\M_\mu^s,d)$ as an isometrically embedded subset. Here $d$ is the distance metric induced by the weak Riemannian metric on $\M_\mu^s$. The space $(X,\d)$ is defined in Theorem~\ref{X}, and Theorem~\ref{vol_fp} is the fixed point theorem for the action of volumorphisms on $\M_\mu^s$. We expect $(X,\d)$ to be the metric completion of $(\M_\mu^s,d)$.

To find a fixed point for the action of a volumorphism on $\M_\mu$, we generalize a mean ergodic theorem and a fixed point theorem to suit our nonlinear setting. These are Theorems~\ref{main_thm} and~\ref{fp_thm}. Mean ergodic theorems consider the convergence of averages of the iterates of the points under the action. In a nonlinear setting there is no obvious notion of average, but on nonpositively curved spaces, such as $\M_\mu^s$, there is a natural generalization of averages~\cite{Jost, Jostb, Karcher}. Averages are also called means or centers of mass in the literature.

The class of metric spaces where we formulate our mean ergodic theorem and fixed point theorem is that of complete global Alexandrov nonpositive curvature spaces, which are also known as (complete) $\mbox{CAT}(0)$ spaces. Alexandrov nonpositive curvature spaces have been studied in a setting similar to ours in~\cite{Jost, Jostb}. In both of the above-mentioned theorems, we assume that the mapping considered is nonexpansive. The mean ergodic theorem additionally assumes that the means of the iterates of the mapping satisfy a certain convexity property.

\section{Mean ergodic theorem and fixed point theorem in a global Alexandrov nonpositive curvature space}
In this section, we formulate and prove a mean ergodic theorem and a fixed point theorem in a nonlinear setting. First we give a brief review of the class of global Alexandrov nonpositive curvature (NPC) spaces we are working on. For details and examples of Alexandrov NPC spaces we refer to~\cite{Jost,Jostb}.

A metric space $(\N,d)$ is said to be a geodesic length space if for any two points $p,q\in N$ there exists a rectifiable curve $\g:[0,1]\to \N$ with $\g(0)=p$ and $\g(1)=q$ and length equal to $d(p,q)$. Such a curve is called a \emph{geodesic}.

\begin{definition}
 A geodesic length space $(\N,d)$ is said to be a global Alexandrov nonpositive curvature (NPC) space if for any three points $p,q,r$ of $\N$ and any geodesic $\g:[0,1]\to \N$ with $\g(0)=p$ and $\g(1)=r$, we have for $0\leq t\leq 1$
\begin{equation*}
 d^2(q,\g(t))\leq (1-t) d^2(q,\g(0))+td^2(q,\g(1))-t(1-t)l(\g)^2.
\end{equation*}
Here $l(\g)$ is the length of the geodesic $\g$. 
\end{definition}
The inequality above is called the Alexandrov NPC inequality. We remark that global Alexandrov NPC spaces are simply connected and that for any two given points there is a unique geodesic connecting the points. See Lemma 2.2.1 of~\cite{Jostb} and the discussion that follows the lemma.

We will also need the concept of convex sets in geodesic length spaces. A subset of a geodesic length space is \emph{convex} if any two points of the subset can be joined by a geodesic whose image is contained in that subset. The convex hull $co(S)$ of a subset $S$ of a geodesic length space is the smallest convex subset of $\N$ containing $S$. 

The convex hull of an arbitrary subset of a geodesic length space need not exist, but global Alexandrov NPC spaces have the property that the convex hull of any set exists~\cite[Lemma 3.3.1]{Jostb}. By that same lemma, we can express the convex hull of a subset $S\subset \N$ as follows. Set $C_0=S$ and define $C_k$ to be the union of all geodesic arcs between points of $C_{k-1}$. We have
\begin{equation}\label{cohull}
 co(S)=\bigcup_{k=0}^\infty C_k.
\end{equation}

We record for future reference that the diameter of a set $S$ and its convex hull $co(S)$ are the same. It holds trivially that $\mbox{diam}(co(S))\geq \mbox{diam}(S)$. Let $\e>0$ and choose $p,q\in co(S)$ such that $d(p,q)+\e=\mbox{diam}(co(S))$. The Alexandrov NPC inequality implies that
\begin{equation}
 d^2(q,\g(t))\leq \max\{d^2(q,\g(0)),d^2(q,\g(1))\}
\end{equation}
for any $q\in \N$, $\g:[0,1]\to\N$ a geodesic and $t\in [0,1]$. By this inequality, it can be seen from the definition of the sets $C_k$ that 
\begin{equation*}
 d^2(p,q)\leq \max_{i\in I}\{d^2(p_i,q_i)\},
\end{equation*}
where $I$ is a finite index set and $p_i,q_i\in C_0=S$. It follows that $d(p,q)$ is bounded from above by $\mbox{diam}(S)$. Thus we also have $\mbox{diam}(S)\geq \mbox{diam}(co(S))$. 

Mean ergodic theorems, in general, are convergence theorems for means of iterates of points under a given action to a limit invariant under the action. We need the concept of mean to proceed. The mean in a vector space is just the arithmetic average of the vectors, but the concept of mean generalizes to many other spaces as follows.

\begin{definition}[Mean]\label{mean_map}
 Let $S=\{S_0,S_1,\ldots, S_{n-1}\}$ be a finite subset of a metric space $\N$ with a metric $d$. The \emph{mean function} of $S$ is the function $F_S$ on $\N$ given by
\begin{equation*}
 F_S(p)=\frac{1}{n}\sum_{i=0}^{n-1}d^2(p,S_i).
\end{equation*}
If there exists a unique minimizer of $F_S$, then we denote 
\begin{equation*}
m(S)=\mbox{ the unique minimizer of } F_S
\end{equation*}
and call $m(S)$ the \emph{mean} of $S$.
\end{definition}
In a complete global Alexandrov NPC space the unique minimizer for $F_S$ exists for all finite subsets $S$ of $\N$ and belongs to the closure $\ol{co}(S)$ of $co(S)$; see Theorem 3.2.1 and Lemma 3.3.4 of~\cite{Jostb}. Thus, in this case, we can also consider the mean as a mapping $S\to \ol{co}(S)$.

We are mainly interested in the means of iterations of points by a given mapping $T$. In this case, we denote the mean function of $n$ iterates of $p\in \N$ by
\begin{equation}\label{F_n}
 F_n(r,p)=\frac{1}{n}\sum_{i=0}^{n-1}d^2(r,T^ip)
\end{equation}
and the mean of $n$ iterates of $p$ is denoted by
\begin{equation*}
 m_n(p)=\mbox{ the unique minimizer of } F_n(\cdot,p).
\end{equation*}

An important class of mappings on $(\N,d)$ we are going to consider is that of nonexpansive mappings. A mapping $T$ from a metric space $(\N,d)$ to itself is called \emph{nonexpansive} if
\begin{equation*}
 d(Tp,Tq)\leq d(p,q)
\end{equation*}
for all $p,q\in \N$. Thus, the class of nonexpansive mappings contains not only contractions, but also isometries.

Convexity plays a crucial role in the formulation and proof of our main theorem. A function $F:\N\to \R$ is said to be convex if for every geodesic $\g:[0,1]\to \N$ the function $F\circ \g:[0,1]\to\R$ is convex. We say that a mapping $T:\N\to \N$ is \emph{distance convex} if for all $n \in \mathbb{N}$ and $q\in \N$ the mapping 
\begin{equation}\label{T_convex}
 d^2(m_n(\cdot),q):\N\to \R^+
\end{equation}
is convex. In a normed vector space, any linear mapping is distance convex, yet this definition seems to be new. However, the proof of our main theorem naturally employs the definition, which suggests that the class of distance convex mappings might be of further interest.

We are now ready to state our main theorem.
\begin{thm}[Mean Ergodic Theorem]\label{main_thm} Let $(\N,d)$ be a complete global Alexandrov NPC space and $T:\N\to\N$ a nonexpansive distance convex mapping. Then, for any $p\in \N$ whose orbit is bounded, and any $q\in \N$, the following are equivalent:
\begin{align*}
  (i) \qquad &Tq=q \mbox{ and } q\in \overline{co}\{p,Tp,T^2p,\ldots\}, \\
 (ii) \qquad  &q=\lim_{n}m_n(p), \\
 (iii) \qquad &q=\mbox{w-}\lim_n m_n(p), \\
  (iv) \qquad &q \mbox{ is a weak cluster point of the sequence } (m_n(p)).
\end{align*}
\end{thm}

Here $\mbox{w-}\lim_n m_n(p)$ refers to \emph{weak convergence} defined in terms of projections as follows. For any $p\in \N$ and any geodesic arc $\g$ in $\N$, there exists a unique point $\pi(p,\g)$ on $\g$ that is closest to $p$. We call $\pi(p,\g)$ the \emph{projection} of $p$ onto $\g$. A point $q\in \N$ is the weak limit of a sequence $(p_n)\subset \N$ if for every geodesic arc through $q$ the sequence $\pi(p_n,\g)$ converges to $q$. Similarly, a point $q\in \N$ is a weak cluster point of a sequence $(p_n)\subset \N$ if, for every neighborhood $U$ of $q$, there are infinitely many natural numbers $n\in \mathbb{N}$ such that $\pi(p_n,\g)\in U$ for every geodesic arc $\g$ through $q$. See Definitions 2.5 and 2.7 of~\cite{Jost} for details on weak convergence. 

Without the assumption that the mapping is distance convex, we still get an interesting weaker version of the theorem.

\begin{thm}[Fixed point theorem]\label{fp_thm} 
Let $(\N,d)$ be a complete global Alexandrov NPC space and $T:\N\to\N$ a nonexpansive mapping. Then, for any $p\in \N$ whose orbit is bounded there exists a fixed point $q$ of $T$ in the subset $\ol{co}\{p,Tp,T^2p,\ldots\}$ of $\N$. 
\end{thm}

The proof of our main theorem is quite lengthy due to the nonstandard framework we are working in. Once the mean ergodic theorem is proved, the fixed point theorem follows easily by using the same techniques. The outline of the proof of Theorem~\ref{main_thm} follows Krengel's proof of a mean ergodic theorem for Ces\'aro bounded operators in Banach spaces~\cite[Theorem 1.1. p.72]{Krengel}. 

Before the proofs of the theorems, we give several auxiliary results to clarify the proofs. The first two statements are general convexity results in global Alexandrov NPC spaces. The statements that follow consider the behavior of minimizers of sequences of convex functions. Then, the results achieved so far are applied to study the behavior of means of iterates of distance convex nonexpansive mappings. The first of the last two auxiliary results shows that projections to convex sets are continuous. The last auxiliary result gives a sufficient condition for the existence of a fixed point of a nonexpansive mapping. 

We begin with a definition.

\begin{definition}[Uniform convexity]
A nonnegative lower semicontinuous function $\psi:\N\to\R^+$ on a geodesic length space $(\N,d)$ is said to be \emph{uniformly convex} if the following quantitative strict convexity condition holds:

For any geodesic $\g:[0,1]\to \N$ and $\e>0$ there exists $\d>0$ such that if
\begin{equation*}
 \psi(\g(\p))\geq \p\psi(\g(0))+\p\psi(\g(1))-\d
\end{equation*}
then
\begin{equation*}
 d(\g(0),\g(1))<\e.
\end{equation*}

A family $\mathcal{F}$ of nonnegative lower semicontinuous functions $\N\to\R^+$ is said to be \emph{equiconvex} if there is a positive number $\d(\e)$ such that the above holds for all $F\in\mathcal{F}$ for any $\d$ smaller than $\d(\e)$. In this case, we call $\d(\e)$ the \emph{modulus of convexity} of $\mathcal{F}$.
\end{definition}
 
We have the following.
\begin{lemma}\label{d2_uni_c}
 Let $(\N,d)$ be a global Alexandrov NPC space. The family of functions $\F=\{d^2(\cdot,p):\N\to \R^+ : \ p\in \N\}$ is equiconvex with $\d(\e)=\e^2/4$.
\end{lemma}
\begin{proof}
Let $p\in\N$, $\e>0$ and $\g:[0,1]\to \N$ be a geodesic. Let $\d<\e^2/4$ and assume that the inequality
\begin{equation}\label{str_con}
 d^2(\g(\p),p)\geq \p d^2(\g(0),p)+\p d^2(\g(1),p)-\d
\end{equation}
holds. In a global Alexandrov NPC the distance between two points is given by the length of the unique geodesic joining the points. Thus the NPC inequality reads 
\begin{equation*}
 \frac{1}{4}d^2(\g(0),\g(1))\leq -d^2(\g(\p),p)+ \p d^2(\g(0),p)+\p d^2(\g(1),p).
\end{equation*}
Together with~\eqref{str_con} we have
\begin{equation*}
 \frac{1}{4}d^2(\g(0),\g(1))\leq \d<\e^2/4.
\end{equation*}
\end{proof}

\begin{cor}\label{set_mean_str_conv}
Let $(\N,d)$ be a global Alexandrov NPC space. The family $\F=\{F_S:  S \mbox{ a finite subset of } \N \}$ of mean functions $F_S$ is equiconvex with $\d(\e)=\e^2/4$.
\end{cor}
\begin{proof}
 Let $S=\{S_0,\ldots,S_{n-1}\}$ be a finite set, let $\e>0$ and $\g:[0,1]\to \N$ be a geodesic. Let $\d<\e^2/4$ and assume that the inequality 
\begin{equation}\label{F_str_con}
F_S(\g(\p))\geq \p F_S(\g(0))+\p F_S(\g(1))-\d
\end{equation}
holds. If the inequality
\begin{equation*}
 d^2(\g(\p),S_i)\geq \p d^2(\g(0),S_i)+\p d^2(\g(1),S_i)-\d
\end{equation*}
is false for all $i=0,\ldots, n-1$, then the inequality~\eqref{F_str_con} is also false. Thus we have for some $0\leq i_0\leq n-1$,
\begin{equation*}
 d^2(\g(\p),S_{i_0})\geq \p d^2(\g(0),S_{i_0})+\p d^2(\g(1),S_{i_0})-\d.
\end{equation*}
Hence we have
\begin{equation*}
 d(\g(0),\g(1))<\e
\end{equation*}
by the previous lemma.
\end{proof}

The proof of our main theorem considers the asymptotic behavior of the mean functions of the iterates of points of a given mapping. In this case, we wish to analyze the asymptotic behavior of the minimizers of these mean functions.  A criterion for the asymptotic minimizers to be close is given by the lemma below.

\begin{lemma}\label{uni_strc_appr}
Let $(\N,d)$ be a geodesic length space. Let $(F_n)$ and $(G_n)$ be two sequences of functions $\N\to \R^+$, for which there exist unique minimizers, $(f_n)$ and $(g_n)$ respectively. Assume also that $(f_n)$ and $(g_n)$ belong to some subset $S$ of $\N$ and that $(F_n)$ is equiconvex with modulus of convexity of $\d(\e)$.

Let $\e>0$ and assume that there exists an $N\in\mathbb{N}$ such that the inequality
\begin{equation*}
 \sup_{p\in S}|F_n(p)-G_n(p)|<\d(\e)
\end{equation*}
holds for all $n\geq N$. Then
\begin{equation*}
 d(f_n,g_n)<\e
\end{equation*}
for all $n\geq N$.
\end{lemma}

\begin{proof}
Denote $\wt{F}_n=F_n-F_n(f_n)$ and $\wt{G}_n=G_n-G_n(g_n)$. Now, the unique minima of $\wt{F}_n$ and $\wt{G}_n$ are zero and the sequence of functions $(\wt{F}_n)$ is equiconvex with modulus of convexity $\d(\e)$. 

Let $\e>0$. By assumption, there is an $N=N(\e)\in\mathbb{N}$ such that for all $n\geq N$
\begin{equation*}
 \sup_{p\in S}|F_n(p)-G_n(p)|<  \d(\e).
\end{equation*}
If $G_n(g_n)-F_n(f_n)\geq 0$, we have
\begin{align*}
 |G_n(g_n)-F_n(f_n)|&=G_n(g_n)-F_n(f_n)\leq G_n(f_n)-F_n(f_n) \\
  &\leq|G_n(f_n)-F_n(f_n)| < \d(\e).
\end{align*}
This implies that
\begin{equation*}
 \sup_{p\in S}|\wt{F}_n(p)-\wt{G}_n(p)|\leq \sup_{p\in S}|F_n(p)-G_n(p)|+ |G_n(g_n)-F_n(f_n)| < 2\d(\e).
\end{equation*}
If $G_n(g_n)-F_n(f_n)\leq 0$, an analogous proof shows that the same conclusion still holds.

Since $\wt{G}_n(g_n)=0$, we get the following inequality
\begin{equation*}
 2\d(\e)> \sup_{p\in S}|\wt{F}_n(p)-\wt{G}_n(p)|_S \geq  |\wt{F}_n(g_n)-\wt{G}_n(g_n)|=\wt{F}_n(g_n).
\end{equation*}
Let $\gamma_n: [0,1]\to \N$ be a geodesic with $\g_n(0)=f_n$ and $\g_n(1)= g_n$. Then, at the midpoint of the geodesic, we have the following estimate 
\begin{equation*}
 \wt{F}_n(\gamma_n(\p))+\d(\e)\geq \d(\e) > \p(\wt{F}_n(\gamma_n(0))+\wt{F}_n(\gamma_n(1))).
\end{equation*}
The equiconvexity of the family $(\wt{F}_n)$ implies
\begin{equation*}
d(f_n,g_n) <\e
\end{equation*}
for all $n\geq N$.
\end{proof}

\begin{lemma}\label{Tdc_impl_close}
 Let $T:\N\to \N $ be a distance convex nonexpansive mapping on an Alexandrov NPC space $(\N,d)$. Assume that $T$ has a bounded orbit $\{p,Tp,T^2p,\ldots\}$ and that $s\in co\{p,Tp,T^2p,\ldots\}$. Then there is an $N\in \mathbb{N}$ such that the means of the first $n$ iterates of $p$ and $s$ satisfy 
\begin{equation*}
 d(m_n(p),m_n(s))<\e
\end{equation*}
for all $n\geq N$.
\end{lemma}
\begin{proof}
Denote the orbit of $p$ by $S_p$. Assume that $s\in co(S_p)$ and let $F_n:\N\times\N \to \R$ be the mean function
\begin{equation*}
 F_n(r,q)=\frac{1}{n}\sum_{i=0}^{n-1}d^2(r,T^iq)
\end{equation*}
of the first $n$ iterates of the second argument. The family $\{F_n(\cdot,s): n\in \N\}$ is equiconvex with modulus of convexity $\d(\e)=\e^2/4$ by Corollary~\ref{set_mean_str_conv}. 

Since $s$ belongs to the bounded set $co(S_p)$ and $T$ is nonexpansive, the orbit of $s$ is bounded: For $i\in \mathbb{N}$, we have 
\begin{align}\label{bded_orbits}
d(s,T^is)&\leq d(s,p)+d(p,T^ip)+ d(T^ip,T^is) \nonumber \\
  &\leq 2d(p,s)+d(p,T^ip)\leq 3~\mbox{diam}(co(S_p)).
\end{align}
Thus the unique minimizers of the functions $\{F_n(\cdot,s)\}$ belong to the bounded set $\ol{co}(S_s) \subset \N$ by the remarks following Definition~\ref{mean_map}.

Recall from Eq.~\ref{cohull} that the convex hull of $S_p$ can be expressed as $\cup_{k=0}^{\infty}C_k$. We prove the claim of the lemma by induction on the index $k$. Assume $s\in C_0=S_p$. Now $s=T^{i_0}p$ for some $i_0$ and we have the following estimate (for $n> i_0$):
\begin{align*}
 F_n(r,s)&=F_n(r,T^{i_0}p)=\frac{1}{n}\sum_{i=0}^{n-1}d^2(r,T^iT^{i_0}p)=\frac{1}{n}\sum_{i=0}^{n-1}d^2(r,T^ip) \\
 &-\frac{1}{n}\sum_{i=0}^{i_0-1}d^2(r,T^ip)+\frac{1}{n}\sum_{i=n}^{i_0+n-1}d^2(r,T^ip).
\end{align*}
By the remarks above, $\ol{co}(S_s)$ is bounded and $\ol{co}(S_p)$ is bounded by assumption. Thus we have that
\begin{equation*}
 \sup_{r\in \ol{co}(S_s)\cup\ol{co}(S_p)}|F_n(r,s)-F_n(r,p)|< \e^2/4
\end{equation*}
for all $n\geq N$ for some sufficiently large $N=N(s)\in \mathbb{N}$. By Lemma~\ref{uni_strc_appr}, there holds
\begin{equation*}
 d(m_n(s),m_n(p))<\e
\end{equation*}
for all $n\geq N(s)$. We have shown that the claim holds for $k=0$ for any $s\in C_0$.

Assume then that the claim holds for some $k>0$ and that $s\in C_{k+1}$. By the definition of $C_{k+1}$, the point $s$ is of the form $s=\g(t_0)$, where $\g:[0,1]\to \N$ is a geodesic with $\g(0), \g(1)\in C_k$ and $t_0\in [0,1]$. The induction assumption for $\g(0)$ and $\g(1)$ together with the assumption that $T$ is distance convex yields
\begin{align*}
 d^2(&m_n(s),m_n(p))\leq t_0 d^2(m_n(\g(0)),m_n(p)) \\
  &+(1-t_0) d^2(m_n(\g(1)),m_n(p))<\e^2
\end{align*}
for all $n\geq \max\{N(\g(0)),N(\g(1))\}$, where $N(\g(0))$ and $N(\g(0))$ are defined by the induction assumption. This completes the induction step.
\end{proof}

Projections to closed convex sets in a complete global Alexandrov NPC space exist by Lemma 2.5. of~\cite{Jost}. The continuity of the projections is given by the following. 
\begin{lemma}
 Let $(\N,d)$ be a complete global Alexandrov NPC space and $C\subset \N$ a closed convex set. Then the projection to $C$ is a continuous mapping.
\end{lemma}
\begin{proof}
Let $\e>0$, and denote by $\pi$ the projection to $C$. Let $p,q\in \N$ and let $\g$ be a geodesic with $\g(0)=\pi(p)$ and $\g(1)=\pi(q)$. Since $C$ is convex, we have $\g(\p)\in C$. Thus, by the definition of $\pi$ there holds
\begin{equation*}
 d(\pi(p),p)\leq d(\g(\p),p) \mbox{ and } d(\pi(q),q)\leq d(\g(\p),q).
\end{equation*}
We also have, by the convexity of $d(\cdot,p)$,
\begin{equation*}
 d(\g(\p),p)\leq \max\{d(\pi(p),p),d(\pi(q),p)\}= d(\pi(q),p),
\end{equation*}
since $d(\pi(p),p)\leq d(\pi(q),p)$. The above inequalities together with the triangle inequality give
\begin{align*}
 d(\g(\p),p)&\leq d(\pi(q),p)\leq d(\pi(q),q)+d(q,p) \\ 
  &\leq d(\g(\p),q)+d(q,p)\leq d(\g(\p),p)+2d(q,p).
\end{align*}
Thus, for any $p$ close enough to $q$, we have $d^2(\pi(q),p)\leq d^2(\g(\p),p)+\e^2/4$ and therefore
\begin{equation*}
 \p d^2(\pi(p),p)+ \p d^2(\pi(q),p) \leq d^2(\g(\p),p)+\e^2/4.
\end{equation*}
The NPC inequality gives
\begin{equation*}
 d(\pi(p),\pi(q))\leq \e
\end{equation*}
for any $p$ close enough to $q$. Thus $\pi$ is continuous.
\end{proof}

The last lemma before the proofs of the main theorems is the following one giving a sufficient condition for a nonexpansive mapping to have a fixed point.
\begin{lemma}\label{cond_for_fpoint}
 Let $(\N,d)$ be a complete global Alexandrov NPC space and let $T:\N\to\N$ be nonexpansive. If a sequence $p_n$ converges weakly to $q$ and $d(p_n,Tp_n)\to 0$, $n\to \infty$, then $Tq=q$.
\end{lemma}

\begin{figure}
 \begin{center}
  \includegraphics[scale=0.6]{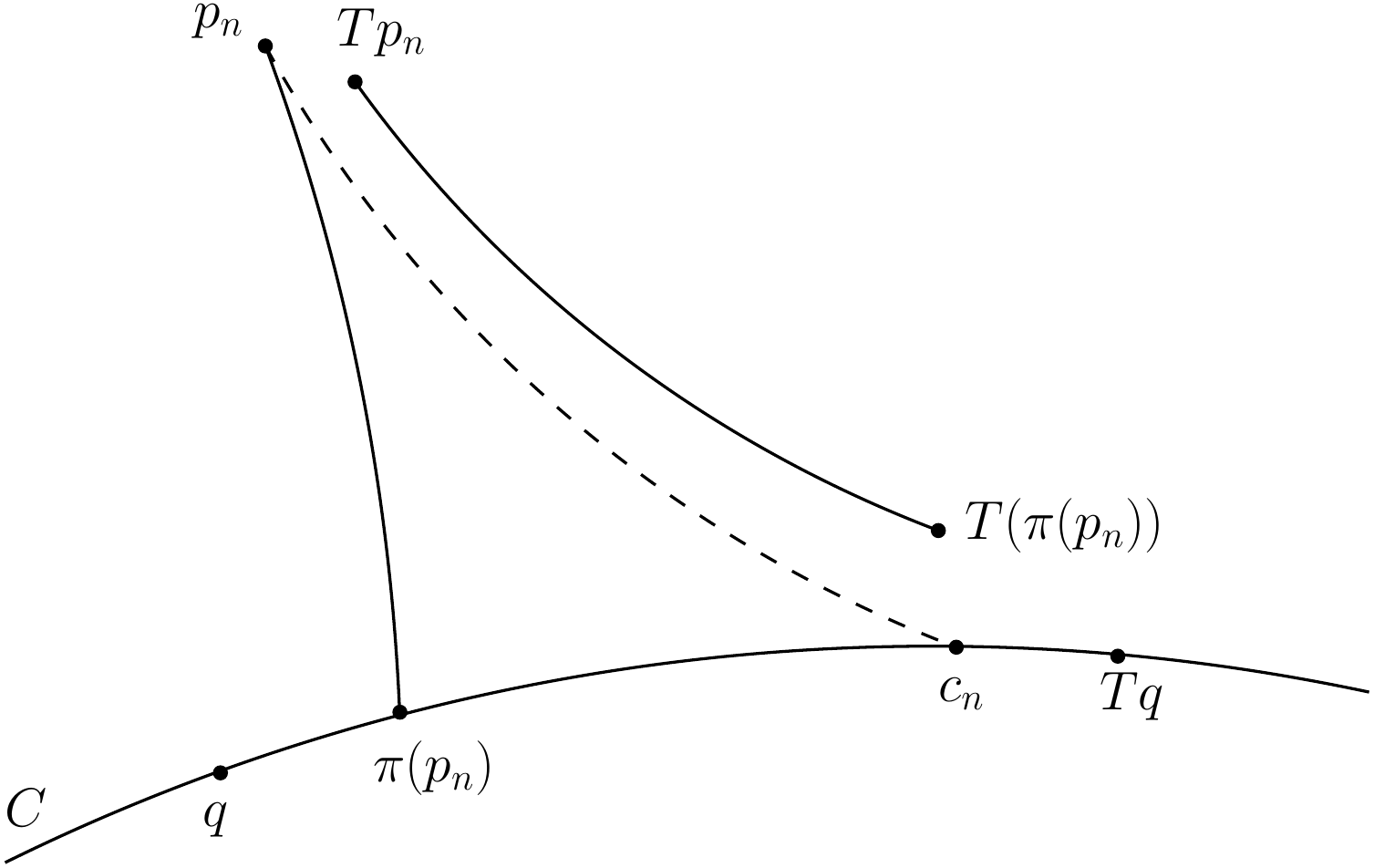}
\caption{\label{jotai} The dashed line is a geodesic arc connecting $p_n$ and $c_n$. The geodesic arc connecting $Tp_n$ and $T(\pi(p_n))$ has length at most $d(p_n,\pi(p_n))$ and it will approach the dashed line as $n$ tends to infinity.}
 \end{center}
\end{figure}
\begin{proof}
The claim follows from geometric considerations illustrated in Figure~\ref{jotai}. Let $C$ be the maximal geodesic containing $q$ and $Tq$ and $\pi:\N\to C$ the projection to this arc. Let us denote the point $\pi (T(\pi(p_n)))\in C$ by $c_n$ and let $\triangle_n$ be the geodesic triangle, whose vertices are $p_n, \pi(p_n)$ and $c_n$. We say that the angle between the geodesic arcs $\overrightarrow{\pi(p_n)p_n}$ and $\overrightarrow{\pi(p_n)c_n}$ is a right angle, because $\pi(p_n)$ minimizes the distance from $p_n$ to $C$ and since $\overrightarrow{\pi(p_n)c_n}\subset C$. We will show, in a sense which will be made precise, that, for any $n$ large enough, the angle between the geodesic arcs $\overrightarrow{c_n p_n}$ and $\overrightarrow{c_n\pi(p_n)}$ is arbitrarily close to a right angle. Thus, as in Euclidean geometry, we will be able to deduce that the length of the side $\overrightarrow{\pi(p_n)c_n}$ of the geodesic triangles $\triangle_n$ will tend to zero as $n\to\infty$. This observation will yield the claim $Tq=q$.

We have $c_n=\pi (T(\pi(p_n)))\to Tq$, $n\to\infty$, because $p_n$ converges weakly to $q$, $\pi$ and $T$ are continuous and $Tq\in C$. The sequences $(T\pi(p_n))$ and $(c_n)$ both converge to $Tq$. Thus
\begin{equation*}
 |d(T(\pi(p_n)),p_n)-d(c_n,p_n)|\to 0, \ n\to\infty.
\end{equation*}
Because $d(Tp_n,p_n)\to 0$, by assumption, we moreover have
\begin{equation}\label{dist_est}
 |d(T(\pi(p_n)),Tp_n)-d(c_n,p_n)|\to 0, \ n\to\infty.
\end{equation}

Let $\g_n$ be a geodesic with $\g_n(0)=\pi(p_n)$ and $\g_n(1)=c_n$. Since $\pi(p_n)$ and $c_n$ belong to $C$, $\g_n(t)\in C$ for all $t\in [0,1]$. The nonexpansiveness of $T$ yields an estimate
\begin{align*}
 d^2&(\g_n(\p),p_n)\geq \p d^2(\pi(p_n),p_n) + \p d^2(Tp_n,T(\pi(p_n)) \\
  &=\p d^2(\g_n(0),p_n) + \p d^2(Tp_n,T(\pi(p_n))
\end{align*}
Here we have also used the fact that $\pi(p_n)$ minimizes the distance from $p_n$ to $C$ and therefore also to the arc of $\g_n$. By the estimate~\eqref{dist_est}, we can choose an $N\in \mathbb{N}$ such that for all $n\geq N$ there holds
\begin{equation*}
 d^2(T(p_n),T(\pi(p_n)))> d^2(c_n,p_n)-\e^2/4=d^2(\g_n(1),p_n)-\e^2/4,
\end{equation*}
which shows that
\begin{equation*}
d^2(\g_n(\p),p_n)> \p d^2(\g_n(0),p_n)+ \p d^2(\g_n(1),p_n)-\e^2/4.
\end{equation*}
Now the equiconvexity of the family $\{d^2(\cdot,p): p\in \N\}$ of functions shows that the length of the geodesic arcs $\overrightarrow{\pi(p_n)c_n}$ of the triangles $\triangle_n$ will tend to zero,
\begin{equation*}
 d(\pi(p_n),c_n)<\e,
\end{equation*}
whenever $n\geq N$. Since $\pi(p_n)\to q$ and $c_n\to Tq$, we have shown that $Tq=q$.
\end{proof}


We are finally set up for the proof of our main theorem.
\begin{proof}[Proof of Theorem~\ref{main_thm}]
$(i)\To (ii)$: Let $\e>0$. Since by assumption $q\in \ol{co}\{p,Tp,T^2p,\ldots\}$, we can find a point $s\in co\{p,Tp,T^2p,\ldots\}$ such that
\begin{equation}\label{approximity}
 2D d(s,q)+ 3d^2(s,q)< \e^2/5,
\end{equation}
where $D$ is the finite diameter of the union of the convex hulls of the orbits $S_q$ and $S_s$ of $q$ and $s$. The fact that $D$ is finite follows from considerations similar to those in~\eqref{bded_orbits}. By Lemma~\ref{Tdc_impl_close}, we have
\begin{equation*}
 d(m_n(s),m_n(p))<\e,
\end{equation*}
whenever $n$ is large enough. 

Let us then estimate the distance
\begin{equation*}
 d(m_n(s),m_n(q)).
\end{equation*}
For this, we estimate the corresponding functions $F_n(\cdot,s)$ and $F_n(\cdot,q)$ (see Eq.~\ref{F_n}). Let $\d>0$. For fixed $r\in \ol{co}(S_q)\cup\ol{co}(S_s)$ we get the following estimate:
\begin{align*}
 F_n(r,s) & =\frac{1}{n}\sum_{i=0}^{n-1}d^2(r,T^i s)\leq \frac{1}{n}\sum_{i=0}^{n-1}(d(r,T^i q) + d(T^i q,T^i s))^2 \\
  & =\frac{1}{n}\sum_{i=0}^{n-1}\l(d^2(r,T^i q) + 2 d(r,T^iq)d(T^i q,T^i s)+d^2(T^i q,T^i s)\r)  \\
  & \leq \frac{1}{n}\sum_{i=0}^{n-1}\l(d^2(r,T^i q) + 2Dd(s,q) +d^2(s,q)\r) < F_n(r,q)+ \e^2/5. 
\end{align*}
Here we have used~\eqref{approximity} and the fact that $T$ is nonexpansive. A similar calculation shows that 
\begin{equation*}
 F_n(r,q)\leq F_n(r,s) + 2Dd(s,q)+3d^2(s,q)< F_n(r,s)+ \e^2/5.
\end{equation*}
Thus we have
\begin{equation*}
 \sup_{r\in \ol{co}(S_q)\cup\ol{co}(S_s)}|F_n(r,s)-F_n(r,q)|\leq \e^2/5<\e^2/4
\end{equation*}
and Lemma~\ref{uni_strc_appr} yields
\begin{equation*}
 d(m_n(s),m_n(q))<\e
\end{equation*}
for all $n$ large enough.

Combining the estimates above shows that
\begin{equation*}
d(q,m_n(p))=d(m_n(q),m_n(p))\leq d(m_n(q),m_n(s))+d(m_n(s),m_n(p))<2\e,
\end{equation*}
whenever $n$ is large enough. Here we have used the assumption $q=Tq$ to deduce $q=m_n(q)$.

$(ii)\To (iii)$: Strong convergence implies weak convergence by the continuity of the projections. $(iii)\To (iv)$: This is obvious.

$(iv)\To (i)$: First we show that the weak cluster point of the sequence $(m_n(p))$ belongs to $\ol{co}\{p,Tp,T^2p,\ldots\}=\ol{co}(S_p)$. Since $q$ is a weak cluster point of a sequence $(m_n(p))$, there exists a subsequence $(m_{n_k}(p))$ converging weakly to $q$ as $k\to\infty$. The sequence $(m_{n_k}(p))$ belongs to $\ol{co}(S_p)$ and in particularly is bounded~\cite[Lemma 3.3.4]{Jostb}. By the version of Mazur's lemma by Jost~\cite[Thm. 2.2]{Jost}, the bounded sequence $(m_{n_k}(p))$ contains a subsequence such that its mean values converge to $q$. But now the mean values of elements of any subsequence of $(m_{n_k}(p))$ belong to the closed set $\ol{co}\{m_{n_k}(p): k\in \mathbb{N}\}\subset \ol{co}(S_p)$. Thus $q\in\ol{co}(S_p)$.

It remains to prove that $q=Tq$. For this let $\e>0$. We use the nonexpansivity of $T$ to show first that, for all $n$ large enough,
\begin{equation*}
 d(Tm_n(p),m_n(p))<\e.
\end{equation*}
We have
\begin{align*}
 F_n(Tm_n(p),p)&=\frac{1}{n}\sum_{i=0}^{n-1}d^2(Tm_n(p),T^ip)\leq \frac{1}{n}\sum_{i=0}^{n-1}d^2(m_n(p),T^ip) \\ 
  &+\frac{1}{n}d^2(T m_n(p),p)-\frac{1}{n}d^2(m_n(p),T^{n-1}p)= F_n(m_n(p),p) \\
  &+\frac{1}{n}d^2(Tm_n(p),p)-\frac{1}{n}d^2(m_n(p),T^{n-1}p).
\end{align*}
By the boundedness of the orbit of $p$, we can choose $N\in\mathbb{N}$ such that, for all $n\geq N$,
\begin{equation*}
 F_n(m_n(p),p)> F_n(Tm_n(p),p)-\e^2/2.
\end{equation*}
For a geodesic $\g:[0,1]\to \N$ with $\g(0)=m_n(p)$ and $\g(1)=Tm_n(p)$, we now have
\begin{equation*}
 F_n(\gamma(\p),p)\geq F_n(m_n(p),p)> \p F_n(\g(0),p) + \p F_n(\g(1),p) -\p \e^2/2.
\end{equation*}
Here the first inequality follows from the fact that $m_n(p)$ is the minimizer of the function $F_n(\cdot,p)$. The equiconvexity of the family of functions $\{F_n(\cdot,p)\}$ now yields
\begin{equation*}
 d(Tm_n(p),m_n(p))<\e
\end{equation*}
whenever $n$ is large enough. 

We have seen that the sequence $(m_{n_k}(p))$ converges weakly to $q$ and 
\begin{equation*}
 d(Tm_{n_k}(p),m_{n_k}(p))\to 0, \ k\to \infty.
\end{equation*}
Thus, by Lemma~\ref{cond_for_fpoint}, we deduce that $Tq=q$.
\end{proof}

From the proof of the theorem we can see that the distance convexity of $T$ was only used to prove (via Lemma~\ref{Tdc_impl_close}) that the means converge strongly to the fixed point. We use this observation to prove Theorem~\ref{fp_thm}.

\begin{proof}[Proof of Theorem~\ref{fp_thm}]
Let $p\in \N$ have bounded orbit. By Theorem 2.1 of~\cite{Jost}, the bounded sequence $(m_n(p))$ contains a subsequence converging weakly to some element $q\in \N$. Thus $q$ is a weak cluster point of the sequence $(m_n(p))$. Now, the exact same argument as in the part $(iv)\To (i)$ of the proof of Theorem~\ref{main_thm} concludes the proof.
\end{proof}

\section{Fixed point for volumorphisms on the space of Riemannian metrics with fixed volume form}
In this section we apply the fixed point theorem to volumorphisms acting on the space of all Sobolev Riemannian metrics having a fixed volume form. We give a natural condition for a fixed point to exist if we allow that the fixed point satisfies only mild regularity assumptions.  We begin by describing the space of Riemannian metrics with a fixed volume form. We refer to~\cite{Ebin,Clarke} for basic results and properties of this space and to~\cite{Freed} for calculations of its geodesics and curvature.

Let $M$ be a smooth compact oriented finite dimensional manifold. We denote by $\M^s$ the set of all Riemannian metrics on $M$, which are of Sobolev class $H^s$. Throughout this section, we will assume $s>n/2$. The space $\M^s$ is an infinite dimensional manifold with a weak Riemannian structure given by
\begin{equation*}
  \langle U,V \rangle_g=\int_M \tr{g^{-1}Ug^{-1}V}dV_g,
\end{equation*}
where $U,V$ are tangent vectors at $g\in\M^s$ and $dV_g$ is the volume form induced by $g$. Tangent vectors of $\M^s$ are symmetric $(0,2)$-tensor fields of Sobolev class $H^s$.

Let $\mu$ be a volume form on $M$. Consider next the subset $\M_\mu^s$ of $\M^s$ consisting of the elements of $\M^s$ whose induced volume form is $\mu$. This subset is an infinite dimensional submanifold of $\M^s$ with the induced inner product
\begin{equation*}
 \langle U,V \rangle_g=\int_M \tr{g^{-1}Ug^{-1}V}d\mu.
\end{equation*}
Here $U$ and $V$, tangent vectors at the point $g$, are traceless (with respect to $g$) symmetric $(0,2)$-tensor fields.  

The geodesics of $\M_\mu^s$ can be given explicitly:
\begin{equation*}\label{can_path}
 g(t)=g\exp{t(g^{-1}A)}.
\end{equation*}
Here $g(0)=g\in \M_\mu^s$ and $\dot{g}(0)=A\in T_g\M_\mu^s$. The geodesic $g(t)$ is of constant speed
\begin{equation*}
 ||\dot{g}(t)||^2_{g(t)}=\int_M\tr{(g^{-1}A)^2}d\mu.
\end{equation*}
Geodesics $g(t)$ exist for all times $t$ and we can easily see that the unique solution $A$ of the equation
\begin{equation*}
 g(t)|_{t=1}=h
\end{equation*}
is
\begin{equation*}
 A=g\log{(g^{-1}h)}.
\end{equation*}
Thus the distance, at least formally, is given by
\begin{equation*}
 d(g,h)=\int_0^1 ||\dot{g}(t)||_{g(t)}dt=\l(\int_M\tr{(g^{-1}A)^2}d\mu\r)^{1/2}.
\end{equation*}
That is,
\begin{equation}\label{P_distance}
 d^2(g,h)=\int_M\tr{(\log{(g^{-1}h)})^2}d\mu.
\end{equation}

The calculation above is formal in the sense that for general infinite dimensional weak Riemannian manifolds it is nontrivial how geodesics relate to the distance function. However, in our case, the exponential mapping $T_g\M_\mu^s\to \M_\mu^s$ is a diffeomorphism onto $\M_\mu^s$ for any $g\in \M_\mu^s$, and therefore the distance between $g$ and $h$ is given by the norm of the tangent vector $A=g\log{(g^{-1}h)}\in T_g\M_\mu^s$ justifying~\eqref{P_distance}. See~\cite[Prop. 2.23, Prop. 2.46]{Clarke} for details. 

We record that $\M_\mu^s$ is indeed a global Alexandrov NPC space. We omit the proof since it is essentially the same as the proof of Theorem~\ref{X} below.
\begin{thm}\label{M_mu_NPC}
 The space $\M_\mu^s$ of Riemannian metrics of Sobolev class $H^s$, $s>n/2$, with a fixed volume form $\mu$ on a compact orientable manifold is a global Alexandrov NPC space.
\end{thm}

A volumorphism is a diffeomorphism preserving a given volume form. We denote by $\mathcal{D}_\mu^{s+1}$ the space of Sobolev $H^{s+1}$ volumorphisms on $M$. The natural action of $\mathcal{D}_\mu^{s+1}$ on $\M_\mu^s$ is given by pullback. A straightforward calculation shows that the action is actually an isometry in the sense of Riemannian geometry and Eq.~\ref{P_distance} shows that it is also an isometry in the metric sense. In particular, the action is nonexpansive and we are, almost, in the setup of our fixed point Theorem~\ref{fp_thm}.

The last needed assumption to apply Theorem~\ref{fp_thm} would be the completeness of $\M_\mu^s$ as a metric space $(\M_\mu^s,d)$. However, $(\M_\mu^s,d)$ is not metrically complete, which is quite expected since we are in a sense considering an $L^2$ inner product in the subset of Sobolev $H^s$ Riemannian metrics on $M$.

The metric completion of the manifold of all Riemannian metrics $\M^s$ with respect to its distance metric is characterized in~\cite{Clarke}. The elements of the completion of $\M^s$ can be identified with measurable semimetrics with finite volume. See~\cite[Thm 5.25.]{Clarke} for details. As $\M_\mu^s$ is a submanifold of $\M^s$, it follows that the $d$-metric completion $\ol{\M}_\mu^s$ of $\M_\mu^s$ is a subset of the metric completion of $\M^s$. The next theorem implies that actually more is true. 

\begin{thm}\label{X}
 Let $\mu$ be a volume form on a smooth oriented compact manifold $M$. Let $X$ be the set of $\mu$-measurable a.e. positive definite symmetric $(0,2)$-tensor fields $g$ on $M$ with volume form agreeing with $\mu$ a.e. and
\begin{equation*}
 \int_M\tr{(\log{(g^{-1}h)})^2}d\mu <\infty,
\end{equation*}
for all $h\in \M_\mu^s$. Then, if $X$ is equipped with the metric
\begin{equation}\label{metric}
 \d(g,h)=\l(\int_M\tr{(\log{(g^{-1}h)})^2}d\mu\r)^{1/2},
\end{equation}
$(X,\d)$ is a complete global Alexandrov NPC space.

The geodesics of $(X,\d)$ are given by the formula
\begin{equation}\label{olP_geo}
 g(t)=g \exp{tg^{-1}A}.
\end{equation}
Here $g\in X$ and $A$ belongs to the set of $(L^2,|\cdot|_g,\mu)$-integrable symmetric $(0,2)$-tensor fields on $M$ satisfying $\Trg{A}{g}=0$ a.e. Moreover, the space of mappings $\mathcal{D}_\mu^{s+1}$ acts by pullback isometrically on $(X,\d)$.
\end{thm}

\begin{proof}
We first show that $X$ equipped with the mapping $\d:X\times X\to [0,\infty)$ is a metric space. This follows from noting that the integrand in the formula~\eqref{metric} of the metric $\d$ is, pointwise in any local coordinates, the square of the distance of matrices in a space isometric to $S:=SL(n,\mathbb{R})/SO(n,\mathbb{R})$. This is because all elements of $X$ have the same volume form a.e., and therefore the coordinate representations of the elements of $X$ are pointwise positive definite symmetric matrices with \emph{equal} determinants a.e.

The observation above yields the following. If $\d(g,h)=0$, then the integrand of~\eqref{metric} equals zero a.e. and consequently $g=h$ a.e. The triangle inequality, and the fact that $\d$ is finite, follow from the Minkowski inequality for $L^2(M,\mu)$ and the triangle inequality on $(S,d_S)$. We conclude that $(X,\d)$ is a metric space. We also see that $\mathcal{D}_\mu^{s+1}$ acts isometrically on $X$.

The metric $d_S$ of $S$ is given by
\begin{align*}
 d^2_S(G,H)&=||\log (H^{1/2}G^{-1}H^{1/2})||^2=\tr{\l(\log(H^{1/2}G^{-1}H^{1/2})\r)^2} \\
  &=\tr{(\log(G^{-1}H))^2}=\tr{(\log(H^{-1}G))^2}
\end{align*}
for $G,H\in S$, see~\cite[Ch. 20]{Iwaniec}. A straightforward calculation shows that a path $\Gamma:\R \to S$ of the form
\begin{equation}\label{S_geo}
 \Gamma(t)=G\exp{t\log{(G^{-1}H)}}
\end{equation}
has on the interval $[0,1]$ the length $l(\Gamma)$,
\begin{equation*}
 l(\Gamma):=\sup\l\{\sum_{i=0}^n d_S(\Gamma(t_i),\Gamma(t_{i-1})): 0=t_0<t_1<\cdots < t_n=1, n\in \mathbb{N} \r\},
\end{equation*}
equal to $d_S(G,H)$. This implies that paths of the form~\eqref{S_geo} are geodesics in $S$. As a Riemannian manifold, $S$ is a complete globally symmetric space of nonpositive sectional curvature and therefore a complete global Alexandrov NPC space, see~\cite[p.11-18, p.55]{Jostb}.

Similarly, if we calculate the length of a path
\begin{equation*}
 \g(t)=g \exp{t\log{(g^{-1}h)}}
\end{equation*}
on an interval $[0,1]$, we see that it equals $\d(g,h)$. Thus, there is a geodesic connecting any two points in $X$. The NPC inequality for $(X,\d)$ is inherited from $(S,d_S)$ via integration of the NPC inequality on $(S,d_S)$. Since $(X,\d)$ is simply connected, it follows that $(X,\d)$ is actually a global Alexandrov NPC space~\cite[Cor. 2.3.2]{Jostb}. More generally we see that paths of the form~\eqref{olP_geo} are geodesics in $(X,\d)$. 

It remains to prove that $(X,\d)$ is complete. The proof of this fact is analogous to showing that $L^p$ spaces are complete, cf.~\cite[Ch. 7.3]{Royden}. For this, let $(g_n)$ be a Cauchy sequence in $(X,\d)$. There is a subsequence of $(g_n)$, which we still denote by $(g_n)$, such that
\begin{align*}
 \sum_{n=0}^{\infty}\d(g_{n},g_{n+1})=s<\infty.
\end{align*}
Thus, by Fatou's lemma and by the Cauchy-Schwarz inequality for $L^2(M,\mu)$, we have that
\begin{align*}
 \int_M &\sum_{n=0}^{\infty}\l(\tr{(\log(g_n^{-1}g_{n+1}))^2}\r)^{1/2}d\mu \\
  &\leq \liminf_{k\to\infty}\sum_{n=0}^{k}\int_M \l(\tr{(\log(g_n^{-1}g_{n+1}))^2}\r)^{1/2}d\mu  \\
  &\leq \liminf_{k\to\infty} \sum_{n=0}^{k} \l( \int_M \tr{(\log(g_n^{-1}g_{n+1}))^2}d\mu \r)^{1/2}\mbox{Vol}_\mu(M)^{1/2} \\
  &=\sum_{n=0}^{\infty}\d(g_n,g_{n+1}) \mbox{Vol}_\mu(M)^{1/2}<\infty.
\end{align*}
Therefore, the integrand $\sum_{n=0}^{\infty}\l(\tr{(\log(g_n^{-1}g_{n+1}))^2}\r)^{1/2}$ above is finite a.e. 

Let $\e>0$. It follows from what we observed that, for a.e. $x\in M$, there is an index $N=N_x\in \mathbb{N}$ such that
\begin{align*}
 d_S&(g_{n}(x),g_{n+k}(x))\leq \sum_{j=n}^{n+k-1}d_S(g_{j}(x),g_{j+1}(x)) \leq \sum_{j=n}^{\infty}d_S(g_{j}(x),g_{j+1}(x)) \\
  &=\sum_{j=n}^{\infty}\l(\tr{(\log(g_j^{-1}(x)g_{j+1}(x)))^2}\r)^{1/2}<\e
\end{align*}
whenever $n\geq N_x$ and for all $k\in \mathbb{N}$. Thus, $(g_n(x))$ is Cauchy a.e. By the metric completeness of $(S,d_S)$, it follows that, for a.e. $x\in M$, the sequence $g_{n}(x)$ tends to some $g(x)$. 

The limit is a symmetric positive semi-definite $(0,2)$ tensor field $g$, defined a.e., and has volume form equalling $\mu$ a.e. The a.e. positiveness of the volume form $\mu$ implies that $g$ is actually positive definite a.e. The fact that $\d(g,h)<\infty$, for all $h\in \M_\mu^s$, follows from Fatou's lemma and the boundedness of Cauchy sequences. We also have $\d(g_n,g)\to 0$ as $n\to\infty$.
\end{proof}

By the discussion of this section we have that $(\M_\mu^s,d)$ is isometrically embedded in $(X,\d)$. It is reasonable to expect that actually $(X,\d)$ is the metric completion of $(\M_\mu^s,d)$. The proof of this, which is equivalent to proving that $\M_\mu^s$ is dense in $(X,\d)$, does not seem to be straightforward however.

We have the main theorem of this section.
\begin{thm}\label{vol_fp}
Let $(X,\d)$ be as in Theorem~\ref{X}. If the action of a volumorphism $\phi\in\mathcal{D}_\mu^{s+1}$ has a bounded orbit in $(X,\d)$ for some $p\in X$, then there exists a fixed point $g$ for the action in the $\d$-closure of the subset $co\{p,Tp,T^2p,\ldots\}$ of $X$. With respect to this fixed point $\phi$ is a Riemannian isometry,
\begin{equation*}
 \phi^*g=g.
\end{equation*}
\end{thm}

\subsection*{Acknowledgements}

The author is supported by the Finnish National Graduate School in Mathematics and its Applications. This work was inspired by the international scientific workshop in honor of Steven Rosenberg's 60th birthday. The author also wishes to thank the referee for valuable comments and interest in this work.

\providecommand{\bysame}{\leavevmode\hbox to3em{\hrulefill}\thinspace}
\providecommand{\href}[2]{#2}

\end{document}